\documentclass[12pt]{amsart}
\usepackage{amssymb}
\usepackage[all]{xy}

\newtheorem{theorem}{Theorem}[section]
\newtheorem{definition}[theorem]{Definition}
\newtheorem{proposition}[theorem]{Proposition}

\begin{document}

\title{Liberation theory for noncommutative homogeneous spaces}

\author{Teodor Banica}
\address{T.B.: Department of Mathematics, Cergy-Pontoise University, 95000 Cergy-Pontoise, France. {\tt teodor.banica@u-cergy.fr}}


\begin{abstract}
We discuss the liberation question, in the homogeneous space setting. Our first series of results concerns the axiomatization and classification of the families of compact quantum groups $G=(G_N)$ which are ``uniform'', in a suitable sense. We study then the quotient spaces of type $X=(G_M\times G_N)/(G_L\times G_{M-L}\times G_{N-L})$, and the liberation operation for them, with a number of algebraic and probabilistic results.
\end{abstract}

\maketitle

\section*{Introduction}

The notion of noncommutative space goes back to an old theorem of Gelfand, stating that any commutative $C^*$-algebra must be of the form $C(X)$, for a certain compact space $X$. In view of this result, one can define the category of ``noncommutative compact spaces'' to be the category of $C^*$-algebras, with the arrows reversed. The category of usual compact spaces embeds then covariantly into this category, via $X\to C(X)$.

Once again by using the Gelfand theorem, each noncommutative space $X$ can be thought of as appearing as ``liberation'' of its classical version $X_{class}$, which is obtained by dividing the corresponding algebra $C(X)$ by its commutator ideal.

We will be interested here in the liberation operation, in the algebraic manifold context. Given a family of noncommutative polynomials $P_i\in\mathbb C<z_1,\ldots,z_N>$, the associated noncommutative manifold $X$, and its classical version $X_{class}$, are given by:
$$\begin{matrix}
X&=&Spec\left(C^*\left(z_1,\ldots,z_N\Big|P_i(z_1,\ldots,z_N)=0\right)\right)\\
\\
\bigcup&&\bigcup\\
\\
X_{class}&=&\left\{(z_1,\ldots,z_N)\in\mathbb C^N\Big|P_i(z_1,\ldots,z_N)=0\right\}
\end{matrix}$$ 

Here the family of polynomials $\{P_i\}$ is assumed to be such that the biggest $C^*$-norm on the universal $*$-algebra $<z_1,\ldots,z_N|P_i(z_1,\ldots,z_N)=0>$ is bounded.

The liberation operation $X_{class}\to X$ can be axiomatized in the quantum group context, the idea being that the category of pairings, which encodes in an abstract way the commutation relations $ab=ba$, must be replaced by a new category of partitions. The theory here, based on Woronowicz's fundamental work in \cite{wo1}, \cite{wo2}, on Wang's free quantum groups \cite{wa1}, \cite{wa2}, on the Weingarten formula \cite{bb+}, \cite{bbc}, \cite{csn}, \cite{wei}, and on the liberation philosophy in free probability theory \cite{bpa}, \cite{nsp}, \cite{spe}, \cite{vdn}, was developed in \cite{bsp}.

In the homogeneous space case, where the general study goes back to \cite{boc}, \cite{po1}, \cite{po2}, some related theory, concerning spaces of type $X=G_N/G_{N-M}$, was developed in \cite{bgo}, and then in \cite{bss}. Such spaces were shown to have a number of interesting features, making them potential candidates for an algebraic manifold extension of \cite{bsp}, provided that the family of compact quantum groups $G=(G_N)$ producing them satisfies:
\begin{enumerate}
\item The ``easiness'' condition in \cite{bsp}, stating that we must have $S_N\subset G_N$, for any $N\in\mathbb N$, with these inclusions being of a certain special type.

\item The ``uniformity'' condition, stating that we must have $G_N\cap U_{N-M}^+=G_{N-M}$, with respect to the standard embedding $U_{N-M}^+\subset U_N^+$.
\end{enumerate}

We will review here this work, by using some new ideas, from \cite{ba1}, \cite{ba2}, \cite{ba3}. On one hand, we will replace the easiness assumption by a condition of type $H_N\subset G_N$, which will allow us to use a twisting parameter $q=\pm1$. On the other hand, we will study more general quotient spaces, depending on parameters $L\leq M\leq N$, as follows:
$$X=(G_M\times G_N)\big/(G_L\times G_{M-L}\times G_{N-L})$$ 

Our main result will be a verification of the Bercovici-Pata liberation criterion, for certain variables associated $\chi\in C(X)$, in a suitable $L,M,N\to\infty$ limit.

There are many questions raised by the present work. Here are some of them:
\begin{enumerate}
\item A first question concerns the full classification of the quantum groups satisfying the above conditions. For some recent advances here, see \cite{fre}, \cite{rwe}, \cite{tw1}, \cite{tw2}.

\item A second question concerns the validity of the quantum isometry group formula $G^+(X^+)=G(X)^+$, in relation with the rigidity results in \cite{chi}, \cite{gjo}.

\item Yet another question regards the possible applications of the present formalism to free probability invariance questions, in the spirit of \cite{csp}, \cite{ksp}.

\item Finally, there are as well several interesting questions in relation with the axiomatization problem for the noncommutative algebraic manifolds \cite{kvv}.
\end{enumerate}

Finally, regarding the general presentation of the paper:

\begin{enumerate}
\item We use Woronowicz's quantum group formalism in \cite{wo1}, \cite{wo2}, with the extra axiom $S^2=id$. A reference here is the book \cite{ntu}. Regarding the noncommutative homogeneous spaces, whose axiomatization is known to run into several difficulties \cite{bss}, \cite{boc}, \cite{dya}, \cite{po1}, \cite{po2}, \cite{sol}, we use here a simplified formalism, best adapted to our examples, that we intend to fully clarify in a forthcoming paper.

\item The present work is a technical continuation of \cite{ba1}, \cite{ba2}, \cite{ba3}, \cite{bgo}, \cite{bss}, with the aim of basically enlarging a list of examples. As already mentioned, the general theory is not available yet.  We have therefore opted for an ``example-first'' presentation. This is actually in tune with the easy quantum group literature, where most of the basic examples were constructed and studied long before \cite{bsp}.
\end{enumerate}

The paper is organized as follows: in 1-2 we discuss the orthogonal and unitary cases, in 3-4 we introduce all the quantum groups that we are interested in, and we study the associated homogeneous spaces, and in 5-6 we discuss probabilistic aspects.

\medskip

\noindent {\bf Acknowledgements.} I would like to thank the referee for a careful reading of the manuscript, and for a number of useful suggestions. This work was partly supported by the ``Harmonia'' NCN grant 2012/06/M/ST1/00169.

\section{Partial isometries}

In this section and in the next one we discuss the construction of the homogeneous spaces that we are interested in, in the case where the underlying groups or quantum groups are $O_N,U_N$, or their twists $\bar{O}_N,\bar{U}_N$, or their free versions $O_N^+,U_N^+$. 

We begin with the classical case. Best is to start as follows:

\begin{definition}
Associated to any integers $L\leq M\leq N$ are the spaces
$$O_{MN}^L=\left\{T:E\to F\ {\rm isometry}\Big|E\subset\mathbb R^N,F\subset\mathbb R^M,\dim_\mathbb RE=L\right\}$$
$$U_{MN}^L=\left\{T:E\to F\ {\rm isometry}\Big|E\subset\mathbb C^N,F\subset\mathbb C^M,\dim_\mathbb CE=L\right\}$$
where the notion of isometry is with respect to the usual real/complex scalar products.
\end{definition}

As a first observation, at $L=M=N$ we obtain the groups $O_N,U_N$. More generally, at $M=N$ we obtain the various components of the semigroups $\widetilde{O}_N,\widetilde{U}_N$ of partial isometries of $\mathbb R^N,\mathbb C^N$, studied in \cite{ba2}, which are by definition given by:
$$\widetilde{O}_N=\bigcup_{L=0}^NO_{NN}^L\qquad\qquad\widetilde{U}_N=\bigcup_{L=0}^NU_{NN}^L$$

Yet another interesting specialization is $L=M=1$. Here the elements of $O_{1N}^1$ are the isometries $T:E\to\mathbb R$, with $E\subset\mathbb R^N$ one-dimensional, and such an isometry is uniquely determined by the element $T^{-1}(1)\in\mathbb R^N$, which must belong to the sphere $S^{N-1}_\mathbb R$. Thus, we have $O_{1N}^1=S^{N-1}_\mathbb R$. Similarly, in the complex case we have $U_{1N}^1=S^{N-1}_\mathbb C$.

In general, the most convenient is to view the elements of $O_{MN}^L,U_{MN}^L$ as rectangular matrices, and to use matrix calculus for their study:

\begin{proposition}
We have identifications of compact spaces
$$O_{MN}^L\simeq\left\{U\in M_{M\times N}(\mathbb R)\Big|UU^t={\rm projection\ of\ trace}\ L\right\}$$
$$U_{MN}^L\simeq\left\{U\in M_{M\times N}(\mathbb C)\Big|UU^*={\rm projection\ of\ trace}\ L\right\}$$
with each partial isometry being identified with the corresponding rectangular matrix.
\end{proposition}

\begin{proof}
We can indeed identify the partial isometries $T:E\to F$ with their corresponding extensions $U:\mathbb R^N\to\mathbb R^M$, $U:\mathbb C^N\to\mathbb C^M$, obtained by setting $U_{E^\perp}=0$, and then identify these latter linear maps $U$ with the corresponding rectangular matrices.
\end{proof}

As an illustration, at $L=M=N$ we recover in this way the usual matrix description of $O_N,U_N$. More generally, at $M=N$ we recover the usual matrix description of $\widetilde{O}_N,\widetilde{U}_N$. See \cite{ba2}. Finally, at $L=M=1$ we obtain the usual description of $S^{N-1}_\mathbb R,S^{N-1}_\mathbb C$.

Now back to the general case, observe that the isometries $T:E\to F$, or rather their extensions $U:\mathbb K^N\to\mathbb K^M$, with $\mathbb K=\mathbb R,\mathbb C$, obtained by setting $U_{E^\perp}=0$, can be composed with the isometries of $\mathbb K^M,\mathbb K^N$, according to the following scheme:
$$\xymatrix@R=15mm@C=15mm{
\mathbb K^N\ar[r]^{B^*}&\mathbb K^N\ar@.[r]^U&\mathbb K^M\ar[r]^A&\mathbb K^M\\
B(E)\ar@.[r]\ar[u]&E\ar[r]^T\ar[u]&F\ar@.[r]\ar[u]&A(F)\ar[u]
}$$

In other words, the groups $O_M\times O_N,U_M\times U_N$ act respectively on $O_{MN}^L,U_{MN}^L$. With the identifications in Proposition 1.2 made, the statement here is:

\begin{proposition}
We have action maps as follows, which are transitive,
$$O_M\times O_N\curvearrowright O_{MN}^L\quad:\quad (A,B)U=AUB^t$$
$$U_M\times U_N\curvearrowright U_{MN}^L\quad:\quad (A,B)U=AUB^*$$
whose stabilizers are respectively $O_L\times O_{M-L}\times O_{N-L}$ and $U_L\times U_{M-L}\times U_{N-L}$.
\end{proposition}

\begin{proof}
We have indeed action maps as in the statement, which are transitive. Let us compute now the stabilizer $G$ of the point $U=(^1_0{\ }^0_0)$. Since the elements $(A,B)\in G$ satisfy $AU=UB$, their components must be of the form $A=(^x_0{\ }^*_a),B=(^x_*{\ }^0_b)$. Now since $A,B$ are both unitaries, these matrices follow to be block-diagonal, and we obtain:
$$G=\left\{(A,B)\Big|A=\begin{pmatrix}x&0\\0&a\end{pmatrix},B=\begin{pmatrix}x&0\\ 0&b\end{pmatrix}\right\}$$

We conclude that the stabilizer of $U=(^1_0{\ }^0_0)$ is parametrized by triples $(x,a,b)$ belonging respectively to $O_L\times O_{M-L}\times O_{N-L}$ and $U_L\times U_{M-L}\times U_{N-L}$, as claimed.
\end{proof}

Finally, let us work out the quotient space description of $O_{MN}^L,U_{MN}^L$:

\begin{theorem}
We have isomorphisms of homogeneous spaces as follows,
\begin{eqnarray*}
O_{MN}^L&=&(O_M\times O_N)/(O_L\times O_{M-L}\times O_{N-L})\\
U_{MN}^L&=&(U_M\times U_N)/(U_L\times U_{M-L}\times U_{N-L})
\end{eqnarray*}
with the quotient maps being given by $(A,B)\to AUB^*$, where $U=(^1_0{\ }^0_0)$.
\end{theorem}

\begin{proof}
This is just a reformulation of Proposition 1.3 above, by taking into account the fact that the fixed point used in the proof there was $U=(^1_0{\ }^0_0)$.
\end{proof}

Once again, the basic examples here come from the cases $L=M=N$ and $L=M=1$, where the quotient spaces at right are respectively $O_N,U_N$ and $O_N/O_{N-1},U_N/U_{N-1}$. In fact, in the general $L=M$ case we obtain the following spaces, considered in \cite{bss}:
\begin{eqnarray*}
O_{MN}^M&=&(O_M\times O_N)/(O_M\times O_{N-M})=O_N/O_{N-M}\\
U_{MN}^M&=&(U_M\times U_N)/(U_M\times U_{N-M})=U_N/U_{N-M}
\end{eqnarray*}

For some further information on these spaces, we refer to \cite{ba2}, \cite{bss}.

\section{Liberations and twists}

We discuss now some noncommutative versions of the above constructions. We use the quantum group formalism of Woronowicz \cite{wo1}, \cite{wo2}, with the extra axiom $S^2=id$. In other words, we consider pairs $(A,u)$ consisting of a $C^*$-algebra $A$, and a unitary matrix $u\in M_N(A)$, such that the following formul\ae\  define morphisms of $C^*$-algebras:
$$\Delta(u_{ij})=\sum_ku_{ik}\otimes u_{kj}\quad,\quad\varepsilon(u_{ij})=\delta_{ij}\quad,\quad S(u_{ij})=u_{ji}^*$$

These morphisms are called comultiplication, counit and antipode. We write $A=C(G)$, and call $G$ a compact matrix quantum group. For full details here, see \cite{ntu}. 

We recall from \cite{ba1}, \cite{bbc} that the compact groups $O_N,U_N$ can be twisted, by replacing the commutation relations $ab=ba,ab^*=b^*a$ between the standard coordinates $u_{ij}(g)=g_{ij}$ with the following commutation/anticommutation relations:
$$ab^\times=\begin{cases}
-b^\times a&{\rm for}\ a\neq b\ {\rm on\ the\ same\ row\ or\ column\ of}\ u\\
b^\times a&{\rm otherwise}
\end{cases}$$

Here $b^\times=b,b^*$, and the precise statement is that these relations, when applied to a matrix $u=(u_{ij})$ which is orthogonal ($u=\bar{u},u^t=u^{-1}$, where $\bar{u}=(u_{ij}^*)$), respectively biunitary ($u^*=u^{-1},u^t=\bar{u}^{-1}$) produce quantum groups $\bar{O}_N,\bar{U}_N$. See \cite{ba1}, \cite{bbc}.

We can liberate $O_{MN}^L,U_{MN}^L$, and then twist them, as follows:

\begin{definition}
Associated to any integers $L\leq M\leq N$ are the algebras
\begin{eqnarray*}
C(O_{MN}^{L+})&=&C^*\left((u_{ij})_{i=1,\ldots,M,j=1,\ldots,N}\Big|u=\bar{u},uu^t={\rm projection\ of\ trace}\ L\right)\\
C(U_{MN}^{L+})&=&C^*\left((u_{ij})_{i=1,\ldots,M,j=1,\ldots,N}\Big|uu^*,\bar{u}u^t={\rm projections\ of\ trace}\ L\right)
\end{eqnarray*}
and their quotients $C(\bar{O}_{MN}^L),C(\bar{U}_{MN}^L)$, obtained by imposing the twisting relations.
\end{definition}

Observe that the above universal algebras are well-defined, because the trace conditions, which read $\sum_{ij}u_{ij}u_{ij}^*=\sum_{ij}u_{ij}^*u_{ij}=L$, show that we have $||u_{ij}||\leq\sqrt{L}$.

We have inclusions between the various spaces constructed so far, as follows:
$$\xymatrix@R=15mm@C=15mm{
U_{MN}^L\ar[r]&U_{MN}^{L+}&\bar{U}_{MN}^L\ar@.[l]\\
O_{MN}^L\ar[r]\ar[u]&O_{MN}^{L+}\ar[u]&\bar{O}_{MN}^L\ar@.[l]\ar@.[u]}$$

Indeed, the inclusions at right follow from definitions, and those at left come from Proposition 1.2, and from the fact that $O_{MN}^L,U_{MN}^L$ are stable by conjugation.

At the level of basic examples now, we first have the following result:

\begin{proposition}
At $L=M=1$ we obtain the diagram
$$\xymatrix@R=15mm@C=15mm{
S^{N-1}_\mathbb C\ar[r]&S^{N-1}_{\mathbb C,+}&\bar{S}^{N-1}_\mathbb C\ar@.[l]\\
S^{N-1}_\mathbb R\ar[r]\ar[u]&S^{N-1}_{\mathbb R,+}\ar[u]&\bar{S}^{N-1}_\mathbb R\ar@.[l]\ar@.[u]}$$
consisting of the liberations and twists of the spheres $S^{N-1}_\mathbb R,S^{N-1}_\mathbb C$.
\end{proposition}

\begin{proof}
We recall from \cite{ba1} that the various spheres are constructed as follows, with the symbol $\times$ standing for ``commutative'', ``twisted'' and ``free'', respectively:
\begin{eqnarray*}
C(S^{N-1}_{\mathbb R,\times})&=&C^*_\times\left(z_1,\ldots,z_N\Big|z_i=z_i^*,\sum_iz_i^2=1\right)\\
C(S^{N-1}_{\mathbb C,\times})&=&C^*_\times\left(z_1,\ldots,z_N\Big|\sum_iz_iz_i^*=\sum_iz_i^*z_i=1\right)
\end{eqnarray*}

Now by comparing with the definition of $O_{1N}^{1\times},U_{1N}^{1\times}$, this proves our claim.
\end{proof}

We have as well the following result, once again making the link with \cite{ba1}:

\begin{proposition}
At $L=M=N$ we obtain the diagram
$$\xymatrix@R=15mm@C=15mm{
U_N\ar[r]&U_N^+&\bar{U}_N\ar@.[l]\\
O_N\ar[r]\ar[u]&O_N^+\ar[u]&\bar{O}_N\ar@.[l]\ar@.[u]}$$
consisting of the liberations and twists of the groups $O_N,U_N$.
\end{proposition}

\begin{proof}
We recall from \cite{ba1} that the various quantum groups are constructed as follows, with the symbol $\times$ standing once again for ``commutative'', ``twisted'' and ``free'':
\begin{eqnarray*}
C(O_N^\times)&=&C^*_\times\left((u_{ij})_{i,j=1,\ldots,N}\Big|u=\bar{u},uu^t=u^tu=1\right)\\
C(U_N^\times)&=&C^*_\times\left((u_{ij})_{i,j=1,\ldots,N}\Big|uu^*=u^*u=1,\bar{u}u^t=u^t\bar{u}=1\right)
\end{eqnarray*}

On the other hand, according to Proposition 1.2 and to Definition 2.1 above, we have the following presentation results:
\begin{eqnarray*}
C(O_{NN}^{N\times})&=&C^*_\times\left((u_{ij})_{i,j=1,\ldots,N}\Big|u=\bar{u},uu^t={\rm projection\ of\ trace}\ N\right)\\
C(U_{NN}^{N\times})&=&C^*_\times\left((u_{ij})_{i,j=1,\ldots,N}\Big|uu^*,\bar{u}u^t={\rm projections\ of\ trace}\ N\right)
\end{eqnarray*}

We use now the standard fact that if $p=aa^*$ is a projection then $q=a^*a$ is a projection too. Together with $Tr(uu^*)=Tr(u^t\bar{u})$ and $Tr(\bar{u}u^t)=Tr(u^*u)$, this gives:
\begin{eqnarray*}
C(O_{NN}^{N\times})&=&C^*_\times\left((u_{ij})_{i,j=1,\ldots,N}\Big|u=\bar{u},\ uu^t,u^tu={\rm projections\ of\ trace}\ N\right)\\
C(U_{NN}^{N\times})&=&C^*_\times\left((u_{ij})_{i,j=1,\ldots,N}\Big|uu^*,u^*u,\bar{u}u^t,u^t\bar{u}={\rm projections\ of\ trace}\ N\right)
\end{eqnarray*}

Now observe that, in tensor product notation, and by using the normalized trace, the conditions at right are all of the form $(tr\otimes id)p=1$, with $p=uu^*,u^*u,\bar{u}u^t,u^t\bar{u}$. We therefore obtain $(tr\otimes\varphi)(1-p)=0$ for any faithful state $\varphi$, and it follows that the projections $p=uu^*,u^*u,\bar{u}u^t,u^t\bar{u}$ must be all equal to the identity, as desired.
\end{proof}

Regarding now the homogeneous space structure of $O_{MN}^{L\times},U_{MN}^{L\times}$, the situation here is more complicated in the twisted and free cases than in the classical case. See \cite{ba1}, \cite{bss}.

The classical results have, however, some partial extensions. In order to formulate a result, we use the standard coaction formalism for the compact quantum groups, as in \cite{bss}. Also, given two noncommutative compact spaces $X,Y$, we define their product $X\times Y$ via the formula $C(X\times Y)=C(X)\otimes C(Y)$, with the tensor product being the maximal one. Finally, we use the standard fact that when $G,H$ are compact matrix quantum groups, then so is their product $G\times H$. For more on these topics, see \cite{bss}, \cite{ntu}.

With these conventions, we have the following result:

\begin{proposition}
The spaces $U_{MN}^{L\times}$ have the following properties:
\begin{enumerate}
\item We have an action $U_M^\times\times U_N^\times\curvearrowright U_{MN}^{L\times}$, given by $u_{ij}\to\sum_{kl}a_{ik}\otimes b_{jl}^*\otimes u_{kl}$.

\item We have a map $U_M^\times\times U_N^\times\to U_{MN}^{L\times}$, given by $u_{ij}\to\sum_{l\leq L}a_{il}\otimes b_{jl}^*$.
\end{enumerate}
Similar results hold for the spaces $O_{MN}^{L\times}$, with all the $*$ exponents removed.
\end{proposition}

\begin{proof}
In the classical case, the transpose of the action map $U_M\times U_N\curvearrowright U_{MN}^L$ and of the quotient map $U_M\times U_N\to U_{MN}^L$ are as follows, where $J=(^1_0{\ }^0_0)$:
\begin{eqnarray*}
\varphi&\to&((A,B,U)\to\varphi(AUB^*))\\
\varphi&\to&((A,B)\to\varphi(AJB^*))
\end{eqnarray*}

But with $\varphi=u_{ij}$ we obtain precisely the formul\ae\  in the statement. The proof in the orthogonal case is similar. Regarding now the free case, the proof goes as follows:

(1) Assuming $uu^*u=u$, with $U_{ij}=\sum_{kl}a_{ik}\otimes b_{jl}^*\otimes u_{kl}$ we have:
\begin{eqnarray*}
(UU^*U)_{ij}
&=&\sum_{pq}\sum_{klmnst}a_{ik}a_{qm}^*a_{qs}\otimes b_{pl}^*b_{pn}b_{jt}^*\otimes u_{kl}u_{mn}^*u_{st}\\
&=&\sum_{klmt}a_{ik}\otimes b_{jt}^*\otimes u_{kl}u_{ml}^*u_{mt}
=\sum_{kt}a_{ik}\otimes b_{jt}^*\otimes u_{kt}=U_{ij}
\end{eqnarray*}

Also, assuming that we have $\sum_{ij}u_{ij}u_{ij}^*=L$, we obtain:
$$\sum_{ij}U_{ij}U_{ij}^*
=\sum_{ij}\sum_{klst}a_{ik}a_{is}^*\otimes b_{jl}^*b_{jt}\otimes u_{kl}u_{st}^*
=\sum_{kl}1\otimes1\otimes u_{kl}u_{kl}^*=L$$

(2) Assuming $uu^*u=u$, with $V_{ij}=\sum_{l\leq L}a_{il}\otimes b_{jl}^*$ we have:
$$(VV^*V)_{ij}=\sum_{pq}\sum_{x,y,z\leq L}a_{ix}a_{qy}^*a_{qz}\otimes b_{px}^*b_{py}b_{jz}^*=\sum_{x\leq L}a_{ix}\otimes b_{jx}^*=V_{ij}$$

Also, assuming that we have $\sum_{ij}u_{ij}u_{ij}^*=L$, we obtain:
$$\sum_{ij}V_{ij}V_{ij}^*=\sum_{ij}\sum_{l,s\leq L}a_{il}a_{is}^*\otimes b_{jl}^*b_{js}=\sum_{l\leq L}1=L$$

By removing all the $*$ exponents, we obtain as well the orthogonal results.

In the twisted case the proof is similar. Let us first discuss the orthogonal case. The twisting relations can be written as follows:
$$u_{ij}u_{pq}=(-1)^{\delta_{ip}+\delta_{jq}}u_{pq}u_{ij}$$

With this formula in hand, the verification of the extra relations goes as follows:

(1) With $U_{ij}=\sum_{kl}a_{ik}\otimes b_{jl}\otimes u_{kl}$ we have, as desired:
\begin{eqnarray*}
U_{ij}U_{pq}
&=&\sum_{klmn}a_{ik}a_{pm}\otimes b_{jl}b_{qn}\otimes u_{kl}u_{mn}\\
&=&\sum_{klmn}(-1)^{\delta_{ip}+\delta_{km}}a_{pm}a_{ik}\otimes(-1)^{\delta_{jq}+\delta_{ln}}b_{qn}b_{jl}\otimes(-1)^{\delta_{km}+\delta_{ln}}u_{mn}u_{kl}\\
&=&\sum_{klmn}(-1)^{\delta_{ip}+\delta_{jq}}a_{pm}a_{ik}\otimes b_{qn}b_{jl}\otimes u_{mn}u_{kl}
=(-1)^{\delta_{ip}+\delta_{jq}}U_{pq}U_{ij}
\end{eqnarray*}

(2) With $V_{ij}=\sum_{l\leq L}a_{il}\otimes b_{jl}$ we have as well, as desired:
\begin{eqnarray*}
V_{ij}V_{pq}
&=&\sum_{l,m\leq L}a_{il}a_{pm}\otimes b_{jl}b_{qm}
=\sum_{l,m\leq L}(-1)^{\delta_{ip}+\delta_{lm}}a_{pm}a_{il}\otimes(-1)^{\delta_{jq}+\delta_{lm}}b_{qm}b_{jl}\\
&=&\sum_{l,m\leq L}(-1)^{\delta_{ip}+\delta_{jq}}a_{pm}a_{il}\otimes b_{qm}b_{jl}=(-1)^{\delta_{ip}+\delta_{jq}}V_{pq}V_{ij}
\end{eqnarray*}

The proof in the unitary case is similar, by adding $*$ exponents where needed.
\end{proof}

Let us examine now the relation between the above maps. In the classical case, given a quotient space $X=G/H$, the associated action and quotient maps are given by:
$$\begin{cases}
a:G\times X\to X&:\quad (g,g'H)\to gg'H\\
p:G\to X&:\quad g\to gH
\end{cases}$$

Thus we have $a(g,p(g'))=p(gg')$. In our context, a similar result holds: 

\begin{theorem}
With $G=G_M\times G_N$ and $X=G_{MN}^L$, where $G_N=O_N^\times,U_N^\times$, we have
$$\xymatrix@R=12mm@C=20mm{
G\times G\ar[r]^m\ar[d]_{id\times p}&G\ar[d]^p\\
G\times X\ar[r]^a&X
}$$
where $a,p$ are the action map and the map constructed in Proposition 2.4.
\end{theorem}

\begin{proof}
At the level of the associated algebras of functions, we must prove that the following diagram commutes, where $\Phi,\pi$ are morphisms of algebras induced by $a,p$:
$$\xymatrix@R=12mm@C=20mm{
C(X)\ar[r]^\Phi\ar[d]_\pi&C(G\times X)\ar[d]^{id\otimes\pi}\\
C(G)\ar[r]^\Delta&C(G\times G)
}$$

When going right, and then down, the composition is as follows:
$$(id\otimes\pi)\Phi(u_{ij})
=(id\otimes\pi)\sum_{kl}a_{ik}\otimes b_{jl}^*\otimes u_{kl}
=\sum_{kl}\sum_{s\leq L}a_{ik}\otimes b_{jl}^*\otimes a_{ks}\otimes b_{ls}^*$$

On the other hand, when going down, and then right, the composition is as follows, where $F_{23}$ is the flip between the second and the third components:
$$\Delta\pi(u_{ij})
=F_{23}(\Delta\otimes\Delta)\sum_{s\leq L}a_{is}\otimes b_{js}^*\\
=F_{23}\left(\sum_{s\leq L}\sum_{kl}a_{ik}\otimes a_{ks}\otimes b_{jl}^*\otimes b_{ls}^*\right)$$

Thus the above diagram commutes indeed, and this gives the result.
\end{proof}

In general, going beyond Theorem 2.5 leads to some non-trivial questions. A first issue comes from the fact that the inclusions $G_L\times G_{M-L}\times G_{N-L}\subset G_M\times G_N$ are not well-defined, in the free case. There are as well some analytic issues, coming from the fact that the maps in Proposition 2.4 (2) are in general not surjective. See \cite{ba1}, \cite{bss}.

\section{Uniform quantum groups}

We discuss in this section a generalization of the above constructions. For this purpose, we first need to axiomatize a suitable class of compact quantum groups, generalizing the classical groups $O_N,U_N$, their twists $\bar{O}_N,\bar{U}_N$, and their free versions $O_N^+,U_N^+$.

Let $P(k,l)$ the set of partitions between an upper row of $k$ points, and a lower row of $l$ points, with each leg colored black or white, and with $k,l$ standing for the corresponding ``colored integers''. We have the following notion, which appears as a straightforward generalization of the corresponding orthogonal notion from \cite{bsp}:

\begin{definition}
A category of partitions is a collection of sets $D=\bigcup_{kl}D(k,l)$, with $D(k,l)\subset P(k,l)$, which contains the identity, and is stable under:
\begin{enumerate}
\item The horizontal concatenation operation $\otimes$.

\item The vertical concatenation $\circ$, after deleting closed strings in the middle.

\item The upside-down turning operation $*$ (with reversing of the colors).
\end{enumerate}
\end{definition}

Here the vertical concatenation operation assumes of course that the colors match. Regarding the identity, the precise condition is that $D(\circ,\circ)$ contains the ``white'' identity $|^{\hskip-1.3mm\circ}_{\hskip-1.3mm\circ}$\,. By using (3) we see that $D(\bullet,\bullet)$ contains the ``black'' identity $|^{\hskip-1.3mm\bullet}_{\hskip-1.3mm\bullet}$\,, and then by using (1) we see that each $D(k,k)$ contains its corresponding (colored) identity.

The basic example of such a category is $P$ itself. Yet another basic example is the category $NC$ of noncrossing partitions. There are of course many other examples. We refer to \cite{bsp} for the uncolored case, and to \cite{tw1}, \cite{tw2} for the general colored case.

As explained in \cite{bsp}, \cite{mal}, \cite{tw2}, such categories produce quantum groups. In this paper, however, we will need a modification of this construction, in order to cover as well the twists. We use for this purpose a number of findings from our recent papers \cite{ba1}, \cite{ba3}.

As explained in \cite{ba1}, \cite{ba3}, in order to cover the twists of the quantum groups in \cite{bsp} we must restrict attention to the categories $D\subset P_{even}$, where $P_{even}\subset P$ is the category of partitions having all blocks of even size. Such partitions act on tensors, as follows:

\begin{definition}
Associated to any $\pi\in P_{even}(k,l)$ are the linear maps
\begin{eqnarray*}
T_\pi(e_{i_1}\otimes\ldots\otimes e_{i_k})&=&\sum_{j:\ker(^i_j)\leq\pi}e_{j_1}\otimes\ldots\otimes e_{j_l}\\
\bar{T}_\pi(e_{i_1}\otimes\ldots\otimes e_{i_k})&=&\sum_{\tau\leq\pi}\varepsilon(\tau)\sum_{j:\ker(^i_j)=\tau}e_{j_1}\otimes\ldots\otimes e_{j_l}
\end{eqnarray*}
where $\{e_i\}$ is the standard basis of $\mathbb C^N$, and $\varepsilon:P_{even}\to\{-1,1\}$ is the signature map.
\end{definition}

Here the kernel of a multi-index $(^i_j)=(^{i_1\ldots i_k}_{j_1\ldots j_l})$ is the partition obtained by joining the sets of equal indices. Thus, the condition $\ker(^i_j)\leq\pi$ simply tells us that the strings of $\pi$ must join equal indices. As for the signature map $\varepsilon:P_{even}\to\{-1,1\}$, this is a canonical extension of the usual signature map $\varepsilon:S_\infty\to\{-1,1\}$, constructed in \cite{ba1}.

Here are a few examples of such linear maps, taken from \cite{ba1}:
$$T_\cap(1)=\bar{T}_\cap(1)=\sum_ie_i\otimes e_i\quad,\quad
T_\cup(e_i\otimes e_j)=\bar{T}_\cup(e_i\otimes e_j)=\delta_{ij}$$
$$T_{\slash\!\!\!\backslash}(e_i\otimes e_j)=e_j\otimes e_i\quad,\quad
\bar{T}_{\slash\!\!\!\backslash}(e_i\otimes e_j)=
\begin{cases}
e_j\otimes e_i&{\rm for}\ i=j\\
-e_j\otimes e_i&{\rm for}\ i\neq j
\end{cases}$$

In general, $T_\pi,\bar{T}_\pi$ can be thought of as coming from a twisting parameter $q=\pm1$. As explained in \cite{ba1}, for $\pi\in NC_{even}$ we have $\tau\leq\pi\implies \varepsilon(\tau)=1$, and so $T_\pi=\bar{T}_\pi$. In general, however, the maps $T_\pi,\bar{T}_\pi$ are different. We refer to \cite{ba1} for full details here.

We have the following ``$q$-easiness'' notion, inspired from \cite{bsp}:

\begin{definition}
A compact quantum group $G\subset U_N^+$ is called quizzy when 
$$Hom(u^{\otimes k},u^{\otimes l})=span\left(\dot{T}_\pi\big|\pi\in D(k,l)\right)$$
for any colored integers $k,l$, for a certain category of partitions $D\subset P_{even}$.
\end{definition}

Here the dot stands for a fixed value of $q=\pm1$, with the maps $T_\pi$ being used at $q=1$, and with the maps $\bar{T}_\pi$ being used at $q=-1$. Also, the ``colored'' tensor powers $u^{\otimes k},u^{\otimes l}$ are defined by tensoring the corepresentations $u^\circ=u$ and $u^\bullet=\bar{u}$.

At the level of basic examples, we have the following result:

\begin{proposition}
The following are quizzy quantum groups, 
$$\xymatrix@R=15mm@C=15mm{
U_N\ar[r]&U_N^+&\bar{U}_N\ar@.[l]\\
O_N\ar[r]\ar[u]&O_N^+\ar[u]&\bar{O}_N\ar@.[l]\ar@.[u]}$$
with $q\in\{-1,1\}$ being by definition $1$ at left, $-1$ at right, and $\pm1$ in the middle.
\end{proposition}

\begin{proof}
As explained in \cite{ba1}, \cite{bsp}, the above quantum groups appear indeed from the following categories of partitions, with $q\in\{-1,1\}$ being as in the statement:
$$\xymatrix@R=16mm@C=15mm{
\mathcal P_2\ar[d]&\mathcal{NC}_2\ar[d]\ar@.[r]\ar[l]&\mathcal P_2\ar@.[d]\\
P_2&NC_2\ar[l]\ar@.[r]&P_2}$$

Here $P_2$ is the set of all pairings, $\mathcal P_2\subset P_2$ is the set of ``matching'' pairings, whose upper and lower strings connect $\circ-\bullet$, and whose through strings connect $\circ-\circ$ or $\bullet-\bullet$, and $NC_2,\mathcal{NC}_2$ are the corresponding subsets of noncrossing pairings. See \cite{ba1}, \cite{bsp}.
\end{proof}

Consider now the group $H_N^s=\mathbb Z_s\wr S_N$, with $s\in\{2,4,\ldots,\infty\}$, which consists of the permutation matrices $\sigma\in S_N$ with nonzero entries multiplied by elements of $\mathbb Z_s$. This group has a free analogue, $H_N^{s+}=\mathbb Z_s\wr_*S_N^+$, constructed in \cite{bb+}, as follows:
$$C(H_N^{s+})=C^*\left((u_{ij})_{i,j=1,\ldots,N}\Big|u_{ij}u_{ij}^*=u_{ij}^*u_{ij}=p_{ij}={\rm magic},u_{ij}^s=p_{ij}\right)$$

Here the ``magic'' condition states that the entries of $p=(p_{ij})$ are projections, summing up to 1 on each row and column, and the last condition,  $u_{ij}^s=p_{ij}$, disappears by definition at $s=\infty$. Observe that the classical version of $H_N^{s+}$ is indeed $H_N^s$. See \cite{bb+}.

The $s=2,\infty$ specializations of $H_N^s,H_N^{s+}$, denoted respectively $H_N,H_N^+$ and $K_N,K_N^+$, are the quantum groups in \cite{bbc}, and their complex analogues. We have:

\begin{proposition}
The following are quizzy quantum groups, at both $q=\pm1$,
$$\xymatrix@R=6mm@C=25mm{
K_N\ar[r]&K_N^+\\
H_N^s\ar[r]\ar[u]&H_N^{s+}\ar[u]\\
H_N\ar[r]\ar[u]&H_N^+\ar[u]}$$
where $H_N^s=\mathbb Z_s\wr S_N$, and where $H_N^{s+}=\mathbb Z_s\wr_*S_N^+$, with $s\in\{2,4,\ldots,\infty\}$. 
\end{proposition}

\begin{proof}
As explained in \cite{ba3}, \cite{bb+}, the above quantum groups appear indeed, with parameter $q=\pm1$ as in the statement, from the following categories of partitions:
$$\xymatrix@R=6mm@C=25mm{
\mathcal P_{even}\ar[d]&\mathcal{NC}_{even}\ar[l]\ar[d]\\
P_{even}^s\ar[d]&NC_{even}^s\ar[l]\ar[d]\\
P_{even}&NC_{even}\ar[l]}$$

Here $P_{even}^s\subset P_{even}$ is the set of partitions having the property that, in each block, the number of white legs equals the number of black legs, modulo $s$, with all legs counted with coefficient $+$ up, and $-$ down. At right we have the subset $NC_{even}^s=P_{even}^s\cap NC$, and the lower and upper objects are the corresponding specializations at $s=2,\infty$.
\end{proof}

We recall that the free complexification $(\widetilde{G},\widetilde{u})$ of a compact matrix quantum group $(G,u)$ is obtained by considering the subalgebra $C(\widetilde{G})\subset C(\mathbb T)*C(G)$ generated by the entries of $\widetilde{u}=zu$, where $z$ is the standard generator of $C(\mathbb T)$. See \cite{rau}.

With this convention, we have the following extra example, from \cite{tw2}:

\begin{proposition}
The quantum group $K_N^{++}=\widetilde{K}_N^+$ appears as an intermediate object
$$K_N^+\subset K_N^{++}\subset U_N^+$$
with both inclusions being proper, and is quizzy, equal to its own twist.
\end{proposition}

\begin{proof}
By composing the canonical inclusion $C(K_N^{++})\subset C(\mathbb T)*C(K_N^+)$ with $\varepsilon*id$ we obtain a morphism $C(K_N^{++})\to C(K_N^+)$ mapping $\widetilde{u}_{ij}\to u_{ij}$, so we have inclusions as in the statement. Since the elements $p_{ij}=\widetilde{u}_{ij}\widetilde{u}_{ij}^*=zu_{ij}u_{ij}^*z^*$ and $q_{ij}=\widetilde{u}_{ij}^*\widetilde{u}_{ij}=u_{ij}^*u_{ij}$ are projections, and are not equal, both these inclusions are proper.

Regarding the easiness claim, this follows from the general theory of the representations of free complexifications \cite{rau}. To be more precise, as explained in \cite{tw2}, the associated category $\mathcal{NC}_{even}^-$ is that of the even noncrossing partitions which, when rotated on one line, have alternating colors in each block. Observe that the inclusions in the statement correspond then to  the inclusions at the partition level, which are as follows:
$$\mathcal{NC}_{even}\supset\mathcal{NC}_{even}^-\supset\mathcal{NC}_2$$

Finally, since $\mathcal{NC}_{even}^-\subset NC_{even}$, the quantum group $K_N^{++}$ equals its own twist.
\end{proof}

The above examples are in fact the only ones that we are interested in, in the classical/twisted and free cases. In order to axiomatize these objects, we use:

\begin{proposition}
For a quizzy quantum group $H_N\subset G_N\subset U_N^+$, coming from a category of partitions $\mathcal{NC}_2\subset D\subset P_{even}$, the following are equivalent:
\begin{enumerate}
\item $D$ is stable by removing blocks.

\item $G_N\cap U_{N-M}^+=G_{N-M}$, for any $M\leq N$.
\end{enumerate}
If these conditions are satisfied, we call both $G=(G_N)$ and $D$ ``uniform''.
\end{proposition}

\begin{proof}
This was proved in \cite{bss} in the orthogonal case, and the proof in general is similar. Assume that we have a subgroup $K\subset U_{N-M}^+$, with fundamental representation $v$, and consider the $N\times N$ matrix $\tilde{v}=diag(v,1_M)$. Then, for any $\pi\in P_{even}$, we have:
$$T_\pi\in Hom(\tilde{v}^{\otimes k},\tilde{v}^{\otimes l})\iff T_{\pi'}\in Hom(v^{\otimes k'},v^{\otimes l'}),\,\forall\pi'\subset\pi$$

With this formula in hand, we deduce that given a subgroup $G\subset U_N^+$, with fundamental representation denoted $u$, the algebra of functions on $K=G\cap U_{N-M}^+$ is given by:
$$C(K)=C(U_{N-M}^+)\Big/\left<T\in Hom(\tilde{v}^{\otimes k},\tilde{v}^{\otimes l}),\,\forall T\in Hom(u^{\otimes k},u^{\otimes l})\right>$$

Thus, we have $G_N\cap U_{N-M}^+=G_{N-M}'$, where $G'=(G_N')$ is the easy quantum group associated to the category $D'$ generated by all the subpartitions of the partitions in $D$. In particular $G_N\cap U_{N-M}^+=G_{N-M}$ for any $M\leq N$ is equivalent to $D=D'$, as claimed.
\end{proof}

Observe that the quantum groups in Propositions 3.4, 3.5, 3.6 are all uniform. We have in fact the following result, where by ``classical/twisted'' and ``free'' we mean $\backslash\hskip-2.1mm/\in D$ and $D\subset NC_{even}$, where $D\subset P_{even}$ is the associated category of partitions:

\begin{theorem}
The uniform classical/twisted and free quantum groups are
$$\xymatrix@R=1mm@C=10mm{
&&U_N,\bar{U}_N\ar@/^/[drr]\\
K_N\ar[rr]\ar@/^/[urr]&&K_N^+\ar[r]&K_N^{++}\ar[r]&U_N^+\\
\\
H_N^s\ar[rr]\ar[uu]&&H_N^{s+}\ar[uu]\\
\\
H_N\ar[rr]\ar[uu]\ar@/_/[drr]&&H_N^+\ar[rr]\ar[uu]&&O_N^+\ar[uuuu]\\
&&O_N,\bar{O}_N\ar@/_/[urr]}$$
where $H_N^s=\mathbb Z_s\wr S_N$, $H_N^{s+}=\mathbb Z_s\wr_*S_N^+$, with $s\in\{2,4,\ldots,\infty\}$, and $K_N^{++}=\widetilde{K}_N^+$. 
\end{theorem}

\begin{proof}
This is a consequence of the recent classification results in \cite{tw1}, \cite{tw2}, the idea being as follows. First, the diagram in the statement being obtained by merging the examples in Propositions 3.4, 3.5, 3.6, all the above quantum groups are quizzy. 

The uniformity condition is clear as well, for each of the quantum groups under consideration. Finally, all these quantum groups are either classical/twisted or free.

In order to prove now the converse, in view of the twisting results in \cite{ba2}, it is enough to deal with the $q=1$ case. So, consider a uniform category of partitions $D\subset P_{even}$, as in Proposition 3.7. We must prove that in the classical/free cases, the solutions are:
$$\xymatrix@R=1.8mm@C=10mm{
&&\mathcal P_2\ar@/_/[dll]\\
\mathcal P_{even}\ar[dd]&&\mathcal{NC}_{even}\ar[ll]\ar[dd]&\mathcal{NC}_{even}^-\ar[l]&\mathcal{NC}_2\ar[l]\ar@/_/[ull]\ar[dddd]\\
\\
P_{even}^s\ar[dd]&&NC_{even}^s\ar[ll]\ar[dd]\\
\\
P_{even}&&NC_{even}\ar[ll]&&NC_2\ar@/^/[dll]\ar[ll]\\
&&P_2\ar@/^/[ull]}$$

To be more precise, in the classical case, where $\backslash\hskip-2.1mm/\in D$, we must prove that the only solutions are the categories $P_2,\mathcal P_2,P_{even}^s$, and that in the free case, where $D\subset NC_{even}$, we must prove that the only solutions are the categories $NC_2,\mathcal{NC}_2,\mathcal{NC}_{even}^-,NC_{even}^s$.

We jointly investigate these two problems. Let $B$ be the set of all possible labelled blocks in $D$, having no upper legs. Observe that $B$ is stable under the switching of colors operation, $\circ\leftrightarrow\bullet$. We have two possible situations, as follows:

(1) $B$ consists of pairings only. Here the pairings in question can be either all labelled pairings, namely $\circ-\circ$, $\circ-\bullet$, $\bullet-\circ$, $\bullet-\bullet$, or just the matching ones, namely $\circ-\bullet$, $\bullet-\circ$, and we obtain here $P_2,\mathcal P_2$ in the classical case, and $NC_2,\mathcal{NC}_2$ in the free case.

(2) $B$ has at least one block of size $\geq 4$. In this case we can let $s\in\{2,4,\ldots,\infty\}$ to be the length of the smallest $\circ\ldots\circ$ block, and we obtain in this way the category $P_{even}^s$ in the classical case, and the categories $\mathcal{NC}_{even}^-,NC_{even}^s$ in the free case. See \cite{tw1}.
\end{proof}

The above occurrence of $K_N^{++}$ is quite an issue, and reminds the ``forgotten series'' from the orthogonal case \cite{bsp}, found later on, in \cite{web}. Note that this forgotten series has disappeared in the present setting, due to the uniformity axiom. In the unitary case, however, it is quite unclear on how to add an extra axiom, as to avoid $K_N^{++}$.

In what follows we will rather ignore $K_N^{++}$, and focus instead on $K_N^+$, known since the Bercovici-Pata computations in \cite{bb+} to be the ``correct'' liberation of $K_N$.

In general now, the full classification of the uniform quizzy quantum groups remains an open problem. For the ongoing program here, we refer to \cite{fre}, \cite{rwe}, \cite{tw1}, \cite{tw2}.

\section{Homogeneous spaces}

In this section we associate noncommutative homogeneous spaces, as in sections 1-2 above, to the quantum groups considered in section 3. We will be mostly interested in the quantum groups from Theorem 3.8, so let us first discuss, with full details, the case of the quantum groups $H_N^s,H_N^{s+}$ appearing there. As in \cite{ba2}, we use:

\begin{definition}
Associated to any partial permutation, $\sigma:I\simeq J$ with $I\subset\{1,\ldots,N\}$ and $J\subset\{1,\ldots,M\}$, is the real/complex partial isometry
$$T_\sigma:span\left(e_i\Big|i\in I\right)\to span\left(e_j\Big|j\in J\right)$$
given on the standard basis elements by $T_\sigma(e_i)=e_{\sigma(i)}$.
\end{definition}

We denote by $S_{MN}^L$ the set of partial permutations $\sigma:I\simeq J$ as above, with range $I\subset\{1,\ldots,N\}$ and target $J\subset\{1,\ldots,M\}$, and with $L=|I|=|J|$. See \cite{ba2}.

In analogy with the decomposition result $H_N^s=\mathbb Z_s\wr S_N$, we have:

\begin{proposition}
The space of partial permutations signed by elements of $\mathbb Z_s$,
$$H_{MN}^{sL}=\left\{T(e_i)=w_ie_{\sigma(i)}\Big|\sigma\in S_{MN}^L,w_i\in\mathbb Z_s\right\}$$
is isomorphic to the quotient space $(H_M^s\times H_N^s)/(H_L^s\times H_{M-L}^s\times H_{N-L}^s)$.
\end{proposition}

\begin{proof}
This follows by adapting the computations in the proof of Proposition 1.3 above. Indeed, we have an action map as follows, which is transitive:
$$H_M^s\times H_N^s\curvearrowright H_{MN}^{sL}\quad:\quad (A,B)U=AUB^*$$

The stabilizer of the point $U=(^1_0{\ }^0_0)$ follows to be the group $H_L^s\times H_{M-L}^s\times H_{N-L}^s$, embedded via $(x,a,b)\to[(^x_0{\ }^0_a),(^x_0{\ }^0_b)]$, and this gives the result.
\end{proof}

In the free case now, the idea is similar, by using inspiration from the construction of the quantum group $H_N^{s+}=\mathbb Z_s\wr_*S_N^+$ in \cite{bb+}. The result here is as follows:

\begin{proposition}
The noncommutative space $H_{MN}^{sL+}$ associated to the algebra
$$C(H_{MN}^{sL+})=C(U_{MN}^{L+})\Big/\left<u_{ij}u_{ij}^*=u_{ij}^*u_{ij}=p_{ij}={\rm projections},u_{ij}^s=p_{ij}\right>$$
has an action map, and is the target of a quotient map, as in Theorem 2.5 above.
\end{proposition}

\begin{proof}
We must show that if the variables $u_{ij}$ satisfy the relations in the statement, then these relations are satisfied as well for the following variables: 
$$U_{ij}=\sum_{kl}a_{ik}\otimes b_{jl}^*\otimes u_{kl}\quad,\quad
V_{ij}=\sum_{l\leq L}a_{il}\otimes b_{jl}^*$$

Since the standard coordinates $a_{ij},b_{ij}$ on the quantum groups $H_M^{s+},H_N^{s+}$ satisfy the relations $xy=xy^*=0$, for any $x\neq y$ on the same row or column of $a,b$, we obtain:
$$U_{ij}U_{ij}^*=\sum_{klmn}a_{ik}a_{im}^*\otimes b_{jl}^*b_{jm}\otimes u_{kl}u_{mn}^*=\sum_{kl}a_{ik}a_{ik}^*\otimes b_{jl}^*b_{jl}\otimes u_{kl}u_{kl}^*$$
$$V_{ij}V_{ij}^*=\sum_{l,r\leq L}a_{il}a_{ir}^*\otimes b_{jl}^*b_{jr}=\sum_{l\leq L}a_{il}a_{il}^*\otimes b_{jl}^*b_{jl}$$

Thus, in terms of the projections $x_{ij}=a_{ij}a_{ij}^*,y_{ij}=b_{ij}b_{ij}^*,p_{ij}=u_{ij}u_{ij}^*$, we have:
$$U_{ij}U_{ij}^*=\sum_{kl}x_{ik}\otimes y_{jl}\otimes p_{kl}\quad,\quad V_{ij}V_{ij}^*=\sum_{l\leq L}x_{il}\otimes y_{jl}$$

By repeating the computation, we conclude that these elements are projections. Also, a similar computation shows that $U_{ij}^*U_{ij},V_{ij}^*V_{ij}$ are given by the same formul\ae.

Finally, once again by using the relations of type $xy=xy^*=0$, we have:
$$U_{ij}^s=\sum_{k_rl_r}a_{ik_1}\ldots a_{ik_s}\otimes b_{jl_1}^*\ldots b_{jl_s}^*\otimes u_{k_1l_1}\ldots u_{k_sl_s}=\sum_{kl}a_{ik}^s\otimes(b_{jl}^*)^s\otimes u_{kl}^s$$
$$V_{ij}^s=\sum_{l_r\leq L}a_{il_1}\ldots a_{il_s}\otimes b_{jl_1}^*\ldots b_{jl_s}^*=\sum_{l\leq L}a_{il}^s\otimes(b_{jl}^*)^s$$

Thus the conditions of type $u_{ij}^s=p_{ij}$ are satisfied as well, and we are done.
\end{proof}

Let us discuss now the general case. We use the Kronecker symbols $\delta_\pi(i)\in\{-1,0,1\}$ from \cite{ba1}, depending on a twisting parameter $q=\pm1$, constructed as follows:
$$\delta_\sigma(i)=\begin{cases}
\delta_{\ker i\leq\sigma}&{\rm (untwisted\ case)}\\
\varepsilon(\ker i)\delta_{\ker i\leq\sigma}&{\rm (twisted\ case)}
\end{cases}$$

With this convention, we have the following result:

\begin{proposition}
The various spaces $G_{MN}^L$ constructed so far appear by imposing to the standard coordinates of $U_{MN}^{L+}$ the relations
$$\sum_{i_1\ldots i_s}\sum_{j_1\ldots j_s}\delta_\pi(i)\delta_\sigma(j)u_{i_1j_1}^{e_1}\ldots u_{i_sj_s}^{e_s}=L^{|\pi\vee\sigma|}$$
with $s=(e_1,\ldots,e_s)$ ranging over all the colored integers, and with $\pi,\sigma\in D(0,s)$.
\end{proposition}

\begin{proof}
According to the various constructions in section 1-2 and in the beginning of this section, the relations defining $G_{MN}^L$ can be written as follows, with $\sigma$ ranging over a family of generators, with no upper legs, of the corresponding category of partitions $D$:
$$\sum_{j_1\ldots j_s}\delta_\sigma(j)u_{i_1j_1}^{e_1}\ldots u_{i_sj_s}^{e_s}=\delta_\sigma(i)$$

We therefore obtain the relations in the statement, as follows:
\begin{eqnarray*}
\sum_{i_1\ldots i_s}\sum_{j_1\ldots j_s}\delta_\pi(i)\delta_\sigma(j)u_{i_1j_1}^{e_1}\ldots u_{i_sj_s}^{e_s}
&=&\sum_{i_1\ldots i_s}\delta_\pi(i)\sum_{j_1\ldots j_s}\delta_\sigma(j)u_{i_1j_1}^{e_1}\ldots u_{i_sj_s}^{e_s}\\
&=&\sum_{i_1\ldots i_s}\delta_\pi(i)\delta_\sigma(i)
=\sum_{\tau\leq\pi\vee\sigma}\sum_{\ker i=\tau}(\pm1)^2\\
&=&\sum_{\tau\leq\pi\vee\sigma}\sum_{\ker i=\tau}1=L^{|\pi\vee\sigma|}
\end{eqnarray*}

As for the converse, this follows by using the relations in the statement, by keeping $\pi$ fixed, and by making $\sigma$ vary over all the partitions in the category.
\end{proof}

In the general case now, where $G=(G_N)$ is an arbitary uniform quizzy quantum group, we can construct spaces $G_{MN}^L$ by using the above relations, and we have:

\begin{theorem}
The spaces $G_{MN}^L\subset U_{MN}^{L+}$ constructed by imposing the relations 
$$\sum_{i_1\ldots i_s}\sum_{j_1\ldots j_s}\delta_\pi(i)\delta_\sigma(j)u_{i_1j_1}^{e_1}\ldots u_{i_sj_s}^{e_s}=L^{|\pi\vee\sigma|}$$
with $\pi,\sigma$ ranging over all the partitions in the associated category, having no upper legs, are subject to an action map/quotient map diagram, as in Theorem 2.5.
\end{theorem}

\begin{proof}
We proceed as in the proof of Proposition 2.4. We must prove that, if the variables $u_{ij}$ satisfy the relations in the statement, then so do the following variables:
$$U_{ij}=\sum_{kl}a_{ik}\otimes b_{jl}^*\otimes u_{kl}\quad,\quad 
V_{ij}=\sum_{l\leq L}a_{il}\otimes b_{jl}^*$$

Regarding the variables $U_{ij}$, the computation here goes as follows:
\begin{eqnarray*}
&&\sum_{i_1\ldots i_s}\sum_{j_1\ldots j_s}\delta_\pi(i)\delta_\sigma(j)U_{i_1j_1}^{e_1}\ldots U_{i_sj_s}^{e_s}\\
&=&\sum_{i_1\ldots i_s}\sum_{j_1\ldots j_s}\sum_{k_1\ldots k_s}\sum_{l_1\ldots l_s}\delta_\pi(i)\delta_\sigma(j)a_{i_1k_1}^{e_1}\ldots a_{i_sk_s}^{e_s}\otimes(b_{j_sl_s}^{e_s}\ldots b_{j_1l_1}^{e_1})^*\otimes u_{k_1l_1}^{e_1}\ldots u_{k_sl_s}^{e_s}\\
&=&\sum_{k_1\ldots k_s}\sum_{l_1\ldots l_s}\delta_\pi(k)\delta_\sigma(l)u_{k_1l_1}^{e_1}\ldots u_{k_sl_s}^{e_s}=L^{|\pi\vee\sigma|}
\end{eqnarray*}

For the variables $V_{ij}$ the proof is similar, as follows:
\begin{eqnarray*}
&&\sum_{i_1\ldots i_s}\sum_{j_1\ldots j_s}\delta_\pi(i)\delta_\sigma(j)V_{i_1j_1}^{e_1}\ldots V_{i_sj_s}^{e_s}\\
&=&\sum_{i_1\ldots i_s}\sum_{j_1\ldots j_s}\sum_{l_1,\ldots,l_s\leq L}\delta_\pi(i)\delta_\sigma(j)a_{i_1l_1}^{e_1}\ldots a_{i_sl_s}^{e_s}\otimes(b_{j_sl_s}^{e_s}\ldots b_{j_1l_1}^{e_1})^*\\
&=&\sum_{l_1,\ldots,l_s\leq L}\delta_\pi(l)\delta_\sigma(l)=L^{|\pi\vee\sigma|}
\end{eqnarray*}

Thus we have constructed an action map, and a quotient map, as in Proposition 2.4 above, and the commutation of the diagram in Theorem 2.4 is then trivial.
\end{proof}

The above results generalize some of the constructions in \cite{ba3}. As explained in \cite{ba3}, there are many interesting questions regarding such spaces, and their quantum isometry groups. In what follows we will advance on some related topics, of probabilistic nature.

\section{Integration theory}

In the remainder of this paper we discuss the integration over $G_{MN}^L$, with a number of explicit formul\ae. Our main result will be the fact that the operations of type $G_{MN}^L\to G_{MN}^{L+}$ are indeed ``liberations'', in the sense of the Bercovici-Pata bijection \cite{bpa}.

The integration over $G_{MN}^L$ is best introduced as follows:

\begin{definition}
The integration functional of $G_{MN}^L$ is the composition
$$tr:C(G_{MN}^L)\to C(G_M\times G_N)\to\mathbb C$$
of the representation $u_{ij}\to\sum_{l\leq L}a_{il}\otimes b_{jl}^*$ with the Haar functional of $G_M\times G_N$.
\end{definition}

Here we use the standard fact, proved by Woronowicz in \cite{wo1}, that any compact quantum group $G$ has a Haar integration functional, $\int_G:C(G)\to\mathbb C$, which is by definition the unique positive unital trace subject to the following invariance relations:
$$\left(\int_G\otimes id\right)\Delta\varphi=\left(id\otimes\int_G\right)\Delta\varphi=\int_G\varphi$$

Observe that in the case $L=M=N$ we obtain the integration over $G_N$. Also, at $L=M=1$ we obtain the integration over the sphere. More generally, at any $L=M$ we obtain the integration over the corresponding row algebra of $G_M$, discussed in \cite{bss}. 

In the general case now, we first have the following result:

\begin{proposition}
The integration functional $tr$ has the invariance property 
$$(id\otimes tr)\Phi(x)=tr(x)1$$
with respect to the coaction map given by $\Phi(u_{ij})=\sum_{kl}a_{ik}\otimes b_{jl}^*\otimes u_{kl}$.
\end{proposition}

\begin{proof}
We restrict the attention to the orthogonal case, the proof in the unitary case being similar. We must check the following formula:
$$(id\otimes tr)\Phi(u_{i_1j_1}\ldots u_{i_sj_s})=tr(u_{i_1j_1}\ldots u_{i_sj_s})$$

Let us compute the left term. This is given by:
\begin{eqnarray*}
X
&=&(id\otimes tr)\sum_{k_rl_r}a_{i_1k_1}\ldots a_{i_sk_s}\otimes b_{j_1l_1}^*\ldots b_{j_sl_s}^*\otimes u_{k_1l_1}\ldots u_{k_sl_s}\\
&=&\sum_{k_rl_r}\sum_{m_r\leq L}a_{i_1k_1}\ldots a_{i_sk_s}\otimes b_{j_1l_1}^*\ldots b_{j_sl_s}^*\int_{G_M}a_{k_1m_1}\ldots a_{k_sm_s}\int_{G_N}b_{l_1m_1}^*\ldots b_{l_sm_s}^*\\
&=&\sum_{m_r\leq L}\sum_{k_r}a_{i_1k_1}\ldots a_{i_sk_s}\int_{G_M}a_{k_1m_1}\ldots a_{k_sm_s}
\otimes\sum_{l_r}b_{j_1l_1}^*\ldots b_{j_sl_s}^*\int_{G_N}b_{l_1m_1}^*\ldots b_{l_sm_s}^*
\end{eqnarray*}

By using now the invariance property of the Haar functionals of $G_M,G_N$, we obtain:
\begin{eqnarray*}
X
&=&\sum_{m_r\leq L}\left(id\otimes\int_{G_M}\right)\Delta(a_{i_1m_1}\ldots a_{i_sm_s})
\otimes\left(id\otimes\int_{G_N}\right)\Delta(b_{j_1m_1}^*\ldots b_{j_sm_s}^*)\\
&=&\sum_{m_r\leq L}\int_{G_M}a_{i_1m_1}\ldots a_{i_sm_s}\otimes\int_{G_N}b_{j_1m_1}^*\ldots b_{j_sm_s}^*\\
&=&\left(\int_{G_M}\otimes\int_{G_N}\right)\sum_{m_r\leq L}a_{i_1m_1}\ldots a_{i_sm_s}\otimes b_{j_1m_1}^*\ldots b_{j_sm_s}^*
\end{eqnarray*}

But this gives the formula in the statement, and we are done.
\end{proof}

We will prove now that $tr$ is in fact the unique positive unital invariant trace on $C(G_{MN}^L)$. For this purpose, we will need the Weingarten formula. We recall from section 4 above that the generalized Kronecker symbols are constructed as follows:
$$\delta_\sigma(i)=\begin{cases}
\delta_{\ker i\leq\sigma}&{\rm (untwisted\ case)}\\
\varepsilon(\ker i)\delta_{\ker i\leq\sigma}&{\rm (twisted\ case)}
\end{cases}$$

With this convention, the integration formula is as follows:

\begin{theorem}
We have the Weingarten type formula
$$\int_{G_{MN}^L}u_{i_1j_1}\ldots u_{i_sj_s}=\sum_{\pi\sigma\tau\nu}L^{|\sigma\vee\nu|}\delta_\pi(i)\delta_\tau(j)W_{sM}(\pi,\sigma)W_{sN}(\tau,\nu)$$
where $W_{sM}=G_{sM}^{-1}$, with $G_{sM}(\pi,\sigma)=M^{|\pi\vee\sigma|}$.
\end{theorem}

\begin{proof}
We make use of the usual quantum group Weingarten formula, for which we refer to \cite{ba1}, \cite{bsp}. By using this formula for $G_M,G_N$, we obtain:
\begin{eqnarray*}
\int_{G_{MN}^L}u_{i_1j_1}\ldots u_{i_sj_s}
&=&\sum_{l_1\ldots l_s\leq L}\int_{G_M}a_{i_1l_1}\ldots a_{i_sl_s}\int_{G_N}b_{j_1l_1}^*\ldots b_{j_sl_s}^*\\
&=&\sum_{l_1\ldots l_s\leq L}\sum_{\pi\sigma}\delta_\pi(i)\delta_\sigma(l)W_{sM}(\pi,\sigma)\sum_{\tau\nu}\delta_\tau(j)\delta_\nu(l)W_{sN}(\tau,\nu)\\
&=&\sum_{\pi\sigma\tau\nu}\left(\sum_{l_1\ldots l_s\leq L}\delta_\sigma(l)\delta_\nu(l)\right)\delta_\pi(i)\delta_\tau(j)W_{sM}(\pi,\sigma)W_{sN}(\tau,\nu)
\end{eqnarray*}

Let us compute now the coefficient appearing in the last formula. Since the signature map takes $\pm1$ values, for any multi-index $l=(l_1,\ldots,l_s)$ we have:
$$\delta_\sigma(l)\delta_\nu(l)=\delta_{\ker l\leq\sigma}\varepsilon(\ker l)\cdot\delta_{\ker l\leq\nu}\varepsilon(\ker l)=\delta_{\ker l\leq\sigma\vee\nu}$$

Thus the coefficient is $L^{|\sigma\vee\nu|}$, and we obtain the formula in the statement.
\end{proof}

We can now derive an abstract characterization of $tr$, as follows:

\begin{proposition}
The integration functional $tr$ constructed above is the unique positive unital $C^*$-algebra trace $C(G_{MN}^L)\to\mathbb C$ which is invariant under the action of $G_M\times G_N$.
\end{proposition}

\begin{proof}
We use the method in \cite{bss}, the point being to show that $tr$ has the following ergodicity property: 
$$\left(\int_{G_M}\otimes\int_{G_N}\otimes id\right)\Phi=tr(.)1$$

We restrict the attention to the orthogonal case, the proof in the unitary case being similar. We must verify that the following holds:
$$\left(\int_{G_M}\otimes\int_{G_N}\otimes id\right)\Phi(u_{i_1j_1}\ldots u_{i_kj_k})=tr(u_{i_1j_1}\ldots u_{i_kj_k})1$$

By using the Weingarten formula, the left term can be written as follows:
\begin{eqnarray*}
X
&=&\sum_{k_1\ldots k_s}\sum_{l_1\ldots l_s}\int_{G_M}a_{i_1k_1}\ldots a_{i_sk_s}\int_{G_N}b_{j_1l_1}\ldots b_{j_sl_s}\cdot u_{k_1l_1}\ldots u_{k_sl_s}\\
&=&\sum_{k_1\ldots k_s}\sum_{l_1\ldots l_s}\sum_{\pi\sigma}\delta_\pi(i)\delta_\sigma(k)W_{sM}(\pi,\sigma)\sum_{\tau\nu}\delta_\tau(j)\delta_\nu(l)W_{sN}(\tau,\nu)\cdot u_{k_1l_1}\ldots u_{k_sl_s}\\
&=&\sum_{\pi\sigma\tau\nu}\delta_\pi(i)\delta_\tau(j)W_{sM}(\pi,\sigma)W_{sN}(\tau,\nu)\sum_{k_1\ldots k_s}\sum_{l_1\ldots l_s}\delta_\sigma(k)\delta_\nu(l)u_{k_1l_1}\ldots u_{k_sl_s}
\end{eqnarray*}

By using now the formula in Theorem 4.5 above, we obtain:
$$X=\sum_{\pi\sigma\tau\nu}L^{|\sigma\vee\nu|}\delta_\pi(i)\delta_\tau(j)W_{sM}(\pi,\sigma)W_{sN}(\tau,\nu)$$

Now by comparing with the formula in Theorem 5.3, this proves our claim.

Assume now that $\tau:C(G_{MN}^L)\to\mathbb C$ satisfies the invariance condition. We have:
\begin{eqnarray*}
\tau\left(\int_{G_M}\otimes\int_{G_N}\otimes id\right)\Phi(x)
&=&\left(\int_{G_M}\otimes\int_{G_N}\otimes\tau\right)\Phi(x)\\
&=&\left(\int_{G_M}\otimes\int_{G_N}\right)(id\otimes\tau)\Phi(x)\\
&=&\left(\int_{G_M}\otimes\int_{G_N}\right)(\tau(x)1)=\tau(x)
\end{eqnarray*}

On the other hand, according to the formula established above, we have as well:
$$\tau\left(\int_{G_M}\otimes\int_{G_N}\otimes id\right)\Phi(x)=\tau(tr(x)1)=tr(x)$$

Thus we obtain $\tau=tr$, and this finishes the proof.
\end{proof}

\section{Probabilistic aspects}

We discuss now the precise computation of the laws of certain linear combinations of coordinates. A set of coordinates $\{u_{ij}\}$ is called ``non-overlapping'' if each horizontal index $i$ and each vertical index $j$ appears at most once. With this convention, we have:

\begin{proposition}
For a sum $\chi_E=\sum_{(ij)\in E}u_{ij}$ of non-overlapping coordinates we have
$$\int_{G_{MN}^L}\chi_E^s=\sum_{\pi\sigma\tau\nu}K^{|\pi\vee\tau|}L^{|\sigma\vee\nu|}W_{sM}(\pi,\sigma)W_{sN}(\tau,\nu)$$
where $K=|E|$ is the cardinality of the indexing set.
\end{proposition}

\begin{proof}
In terms of $K=|E|$, we can write $E=\{(\alpha(i),\beta(i))\}$, for certain embeddings $\alpha:\{1,\ldots,K\}\subset\{1,\ldots,M\}$ and $\beta:\{1,\ldots,K\}\subset\{1,\ldots,N\}$. In terms of these maps $\alpha,\beta$, the moment in the statement is given by:
$$M_s=\int_{G_{MN}^L}\left(\sum_{i\leq K}u_{\alpha(i)\beta(i)}\right)^s$$

By using the Weingarten formula, we can write this quantity as follows:
\begin{eqnarray*}
M_s
&=&\int_{G_{MN}^L}\sum_{i_1\ldots i_s\leq K}u_{\alpha(i_1)\beta(i_1)}\ldots u_{\alpha(i_s)\beta(i_s)}\\
&=&\sum_{i_1\ldots i_s\leq K}\sum_{\pi\sigma\tau\nu}L^{|\sigma\vee\nu|}\delta_\pi(\alpha(i_1),\ldots,\alpha(i_s))\delta_\tau(\beta(i_1),\ldots,\beta(i_s))W_{sM}(\pi,\sigma)W_{sN}(\tau,\nu)\\
&=&\sum_{\pi\sigma\tau\nu}\left(\sum_{i_1\ldots i_s\leq K}\delta_\pi(i)\delta_\tau(i)\right)L^{|\sigma\vee\nu|}W_{sM}(\pi,\sigma)W_{sN}(\tau,\nu)
\end{eqnarray*}

But, as explained in the proof of Theorem 5.3, the coefficient on the left in the last formula equals $K^{|\pi\vee\tau|}$. We therefore obtain the formula in the statement.
\end{proof}

We can further advance in the classical/twisted and free cases, where the Weingarten theory for the corresponding quantum groups is available from \cite{ba1}, \cite{bb+}, \cite{bsp}. The result here, which justifies our various ``liberation'' claims, is as follows:

\begin{theorem}
In the context of the liberation operations $O_{MN}^L\to O_{MN}^{L+}$, $U_{MN}^L\to U_{MN}^{L+}$, $H_{MN}^{sL}\to H_{MN}^{sL+}$, the laws of the sums of non-overlapping coordinates,
$$\chi_E=\sum_{(ij)\in E}u_{ij}$$
are in Bercovici-Pata bijection, in the $|E|=\kappa N,L=\lambda N,M=\mu N,N\to\infty$ limit.
\end{theorem}

\begin{proof}
We use the general theory in \cite{ba1}, \cite{bb+}, \cite{bsp}. According to Proposition 6.1 above, in terms of $K=|E|$, the moments of the variables in the statement are given by:
$$M_s=\sum_{\pi\sigma\tau\nu}K^{|\pi\vee\tau|}L^{|\sigma\vee\nu|}W_{sM}(\pi,\sigma)W_{sN}(\tau,\nu)$$

We use now two standard facts, namely the fact that in the $N\to\infty$ limit the Weingarten matrix $W_{sN}$ is concentrated on the diagonal, and the fact that we have $|\pi\vee\sigma|\leq\frac{|\pi|+|\sigma|}{2}$, with equality precisely when $\pi=\sigma$. See \cite{bsp}. In the regime $K=\kappa N,L=\lambda N,M=\mu N,N\to\infty$ from the statement, we therefore obtain:
\begin{eqnarray*}
M_s
&\simeq&\sum_{\pi\tau}K^{|\pi\vee\tau|}L^{|\pi\vee\tau|}M^{-|\pi|}N^{-|\tau|}\\
&\simeq&\sum_\pi K^{|\pi|}L^{|\pi|}M^{-|\pi|}N^{-|\pi|}\\
&=&\sum_\pi\left(\frac{\kappa\lambda}{\mu}\right)^{|\pi|}
\end{eqnarray*}

In order to interpret this formula, we use general theory from \cite{bb+}, \cite{nsp}:

(1) For $G_N=O_N,\bar{O}_N/O_N^+$, the above variables $\chi_E$ follow to be asymptotically Gaussian/semicircular, of parameter $\frac{\kappa\lambda}{\mu}$, and hence in Bercovici-Pata bijection.

(2) For $G_N=U_N,\bar{U}_N/U_N^+$ the situation is similar, with $\chi_E$ being asymptotically complex Gaussian/circular, of parameter $\frac{\kappa\lambda}{\mu}$, and in Bercovici-Pata bijection. 

(3) Finally, for $G_N=H_N^s/H_N^{s+}$, the variables $\chi_E$ are asymptotically Bessel/free Bessel of parameter $\frac{\kappa\lambda}{\mu}$, and once again in Bercovici-Pata bijection.  
\end{proof}

The convergence in the above result is of course in moments, and we do not know whether some stronger convergence results can be formulated. Nor do we know whether one can use linear combinations of coordinates which are  more general than the sums $\chi_E$ that we consider. These are interesting questions, that we would like to raise here.

\end{document}